\documentclass[11pt]{article}
\usepackage{amssymb,amsmath,amsthm,mathdesign}
\usepackage{pdfsync}
\usepackage{color}
\usepackage[colorlinks]{hyperref}

\newtheorem{theorem}{Theorem}[section]
\newtheorem{lemma}[theorem]{Lemma}
\newtheorem{proposition}[theorem]{Proposition}

\newtheorem{assumption}[theorem]{Assumption}
\newtheorem{condition}[theorem]{Condition}
\newtheorem{remark}[theorem]{Remark}

\numberwithin{equation}{section}

\newcommand{\be}{\begin{equation}}
\newcommand{\ee}{\end{equation}}

\newcommand{\bes}{\begin{equation*}}
\newcommand{\ees}{\end{equation*}}

\def\E{\bE}

\def\P{\bP} %% environment measure
\def\bE{\mathbb{E}}

\newcommand{\R}{\mathbf{R}}

\renewcommand{\d}{{\rm d}}

\renewcommand{\geq}{\geqslant}
\renewcommand{\leq}{\leqslant}

\renewcommand{\P}{\mathrm{P}}
\def\m1{\mathbf{1}}

\allowdisplaybreaks[1]

\title{Non-existence results for stochastic wave equations in one dimension}
\author{Mohammud Foondun\\ University of Strathclyde \and Eulalia Nualart \\Pompeu Fabra University\\}
\date{}
\begin{document}
\maketitle

\begin{abstract}
The purpose of this paper is to extend recent results of \cite{BonderG} and \cite{Foondun-Nualart} for the stochastic heat equation to the stochastic wave equation given by
\begin{equation*}
\left|\begin{aligned}
&\frac{\partial^2 u(t, x)}{\partial t^2}=\frac{\partial^2 u(t,\,x)}{\partial x^2}+ \sigma(u(t,x)) \dot{W}(t,\,x)+b(u(t,x)), \quad x \in D \setminus \partial  D, 
\; t>0,\\
&u(0,x)=u_0(x),\quad \frac{\partial }{\partial t}u(0,x) =v_0(x), \quad x \in D,
\end{aligned}\right.
\end{equation*}
where $\dot{W}$ is space-time white noise, $\sigma$ is a real-valued globally Lipschitz function but $b$ is assumed to be only locally Lipschitz continuous.  Three types of domain conditions are studied: $D=[0,1]$ with homogeneous Dirichlet boundary conditions, $D=[0, 2\pi]$ with periodic boundary conditions, and $D=\R$. 
 Then, under suitable conditions, the following integrability condition 
 $$
 \int_\alpha^{\infty} \frac{1}{[\beta^2+2\int_\alpha^s b(r) \, \d r]^{1/2}} \d s <\infty,\qquad \text{ for some } \alpha>0 \text{ and }  \beta > 0, 
 $$
is studied in relation to non-existence of global solutions.
\end{abstract}

\section{Introduction}
Consider the following stochastic heat equation 
\begin{equation}\label{dirichlet}
\left|\begin{aligned}
\frac{\partial u(t,\,x)}{\partial t}&= \frac{\partial^2 u(t,\,x)}{\partial x^2}+\sigma(u(t,\,x))\dot{W}(t,\,x)+b(u(t,\,x)),\quad x\in [0,\,1],\; t>0,\\
u(0,\,x)&=u_0(x),
\end{aligned}\right.
\end{equation}
with homogeneous Dirichlet boundary conditions, where $\sigma:\R \rightarrow \R$ is a globally Lipschitz function and $b:\R \rightarrow \R$ is a locally Lipschitz function.  The initial condition $u_0$ is assumed to  be nonnegative and continuous and $\dot{W}$ is a space-time Gaussian white noise.  In \cite{BonderG} Bonder and Groisman show that when  $\sigma(x)$ is a positive constant, the solution to (\ref{dirichlet}) blows up in finite time whenever $b$ is nonnegative, convex, and satisfies the following well-known {\it Osgood condition} for ordinary  differential equations: for some $a>0$ 
\begin{equation}\label{osgood}
\int_{a}^{\infty} \frac{1}{b(s)}\,\d s<\infty,
\end{equation}
where $1/0=\infty$. The recent results of  \cite{Dal-Khos-Zhang} and \cite{Foondun-Nualart} imply that  condition  (\ref{osgood}) is a necessary as well as sufficient condition for blow-up. More precisely,  Theorem 1.4 in \cite{Dal-Khos-Zhang} shows that if $u_0$ is H\"older continuous, $\vert \sigma(x)\vert=O(\vert x\vert (\log \vert x\vert)^{1/4})$ and $\vert b(x)\vert=O(\vert x\vert \log \vert x\vert)$ as $\vert x \vert \rightarrow \infty$, then there exists a global solution to equation (\ref{dirichlet}). In \cite{Foondun-Nualart}, it is shown that if $\sigma(x)$ is a positive constant and $b$ is nonegative and nondecreasing on $(0,\,\infty)$, then $b$ satisfies the Osgood condition \eqref{osgood} provided that the solution to \eqref{dirichlet} blows up in finite time with positive probability. Moreover, Bonder and Groisman's result is also derived for the case  where $[0,1]$ is replaced by the real line. Namely, 
if $\sigma$ is bounded and $b$ is nonegative, nondecreasing on $(0,\,\infty)$, and satisfies the Osgood condition \textnormal{(\ref{osgood})}, then almost surely, there is no global solution to equation \eqref{dirichlet} in the real line.

\vskip 12pt

The aim of this paper is to find analogous results for stochastic wave equations of the form
\begin{equation}\label{main-eq1}
\left|\begin{aligned}
&\frac{\partial^2 u(t, x)}{\partial t^2}=\frac{\partial^2 u(t,\,x)}{\partial x^2}+ \sigma(u(t,x)) \dot{W}(t,\,x)+b(u(t,x)), \quad x \in D \setminus \partial  D, 
\; t>0,\\
&u(0,x)=u_0(x),\quad \frac{\partial }{\partial t}u(0,x) =v_0(x), \quad x \in D,
\end{aligned}\right.
\end{equation}
where the initial conditions $u_0$ and $v_0$ are real-valued continuous functions, $\dot{W}$ is space-time Gaussian white noise, and $\sigma, b:\R \rightarrow \R$ are globally and locally Lipschitz functions, respectively.  We  consider three different cases for the domain $D$ with the associated  boundary conditions:
\begin{itemize}
\item {\bf  Case 1}: $D=[0,1]$ with homogeneous Dirichlet boundary conditions $$u(t,0)=u(t,1)=0,\quad   t> 0.$$
\item {\bf Case 2}: $D=I$, where $I$ denotes the unit circle, with periodic  boundary conditions
$$u(t,0)=u(t,2\pi),\quad  \frac{\partial }{\partial x}u(t,0)=\frac{\partial }{\partial x}u(t,2\pi), \quad  t> 0.$$
\item {\bf  Case 3}: $D=\R$ with the extra condition that $u_0$ and $v_0$ are bounded.
\end{itemize}

Following Walsh \cite{Walsh}, a local random field solution to \eqref{main-eq1} is a jointly measurable and adapted process $u=\{u(t,x)\}_{(t,x) \in \R_+ \times D}$ satisfying the following integral equation
\begin{equation}\label{mild1} \begin{split}
u(t,&x)=\int_D G_i(t,x,y) v_0(y) \, \d y+\frac{\partial}{\partial t} \left(\int_D G_i(t,x,y) u_0(y) \, \d y \right)\\
&+\int_0^t\int_D G_i(t-s, x,y) \sigma(u(s,y))\,W(\d s\,\d y) +\int_0^t\int_D G_i(t-s, x,y)b(u(s,y))\,\d s\,\d y \quad \text{a.s.}
\end{split}
\end{equation}
for all $t\in(0,\,\tau)$, where $\tau$ is some stopping time. If we can take $\tau=\infty$, then the local solution is also a global one. Here $G_i(t,x,y)$ is the fundamental solution or Green function of the wave equation  for Cases $i=1,2$ and 3, that is, 
\begin{equation*}
\left|\begin{aligned}
&\frac{\partial^2 }{\partial t^2}G_i(t,x,y)=\frac{\partial^2 }{\partial x^2} G_i(t,x,y), \quad x,y \in D \setminus \partial  D, 
\; t>0,\\
&G_i(0,x,y)=0, \quad \frac{\partial}{\partial t}G_i(0,x,y)=\delta_0(x-y),\quad G_i(t,x,y)\vert_{x \in \partial D}=0.
\end{aligned}\right.
\end{equation*}
It is well-known that the $G_i(t,x,y)$'s have  following expressions
\begin{equation} \label{G1}
G_1(t,x,y):= \sum_{n=1}^{\infty} \frac{\sin(n\pi t)}{n\pi} \varphi_n(x) \varphi_n(y), \quad x,y \in [0,1],
\end{equation}
where $\varphi_n(x)=\sqrt{2} \sin(n\pi x)$, $n \geq  1$ is a complete orthonormal system of $L^2([0,1])$,
$$
G_2(t,x,y)=G_2(t,x-y):= \sum_{n \in \mathbb{Z}} \frac 12 {\bf 1}_{\{ \vert x-y +2n\pi\vert <t\}}, \quad x,y \in I,
$$
and
$$
G_3(t,x,y)=G_3(t,x-y):=\frac 12 {\bf 1}_{\{ \vert x-y \vert <t\}}, \quad x,y \in \R.
$$
In the Appendix, we give some properties of these kernels that will be useful in the sequel.
For more information about the kernels $G_1$ and $G_2$, see for e.g. \cite{QS06} and  \cite{Muellerwave2}, respectively.   

Local existence and uniqueness for Cases 1 and 2 is known and follow from say \cite[Proposition II.3]{CarmonaNualart}  after a truncation procedure. We set 
\begin{equation} \label{tau}
\tau:=\sup\left\{ t>0: \sup_{x\in D}|u(t,\,x)|<\infty\right\},
\end{equation}
where $\sup \emptyset:=-\infty$. In fact, it suffices to use the same truncation argument as for the  heat equation in $[0,1]$ explained for e.g. in  \cite{Dal-Khos-Zhang, Foondun-Nualart}. If $\P(\tau<\infty)>0$, then we say that the solution blows up in finite time with positive probability and if  $\P(\tau<\infty)=1$, we say that the solution blows up in finite time almost surely. For Case 3, much less is known about local existence, see for example the introduction in \cite{Millet-Sanz}. 
For instance, when $b$ is a polynomial and $\sigma$ is continuous and the Lipschitz  constant grows polynomially, \cite{Chow3} shows the existence of a local solution to equation (\ref{main-eq1})
in the  Sobolev space $H^1(\R^d)$, $d \leq 3$, when the noise is white in time and spatially correlated. Concerning global existence, in \cite{Muellerwave}, the  case $b=0$ and $\vert \sigma(x) \vert \geq \vert  x \vert (\log \vert x \vert)^{1/2-\epsilon}$, for $\epsilon>0$ is considered, showing the  existence of a global unique solution. In the recent paper \cite{Millet-Sanz}, the authors study the compact support case in spatial dimension $d \in \{1,2,3\}$, and show that if  $\vert b(x)\vert=O(\vert x\vert (\log \vert x\vert)^{\delta})$ and $\vert \sigma(x)\vert=O(\vert x\vert (\log \vert x\vert)^{a})$ as $\vert x \vert \rightarrow \infty$, with  $\delta  \geq 2a>0$, then there is a unique global solution provided that $\delta<2$. 

However, a general integral condition for non-existence of global solutions as obtained in \cite{BonderG} and \cite{Foondun-Nualart} for the stochastic heat equation has not been addressed in the literature for the stochastic wave equation. The purpose of this paper is to contribute to filling this gap.  
\vskip 12pt
We will always work under the following assumption on the drift coefficient.
\begin{assumption}\label{b}
The function $b$ is locally Lipschitz,  nonnegative and nondecreasing.
\end{assumption}

We will also need the following integrability condition.
\begin{condition}\label{B}
For some $\alpha>0$ and $\beta > 0$,  
\begin{equation}\label{integrability}
T(\alpha,\,\beta):= \int_\alpha^{\infty} \frac{1}{[\beta^2+2\int_\alpha^s b(r) \, \d r]^{1/2}} \d s <\infty,
\end{equation} 
where $1/0=\infty$.
\end{condition} 
Observe that if $T(\alpha,\,\beta)$ is finite for some $\alpha>0$ and $\beta>0$, then it is also finite for all $\alpha>0$ and $\beta>0$, see Remark \ref{22} below.

Our first result concerns the bounded domain $[0,1]$ with homogeneous Dirichlet boundary conditions. In the special case that $\sigma$ is constant, it says that Condition \ref{B} is both necessary and sufficient for non-existence of global solutions to \eqref{main-eq1}. 

\begin{theorem}\label{Bonder-Groissman}
Suppose that Assumption \ref{b} holds. Consider equation \eqref{main-eq1} for Case 1. If $\sigma$ is bounded and the solution blows up in finite time with positive probability then $b$ satisfies Condition \ref{B}. On the other hand, if $\sigma$ is a positive constant, $u_0$ and $v_0$  are nonnegative, and Condition \ref{B} holds, then the solution blows up in finite time with a positive probability provided that $b$ is also convex.
\end{theorem}

For second result, we consider the unit circle with periodic boundary condition. In this case, under Condition \ref{B}, we are able to prove almost sure blow-up as opposed to blow-up with positive probability.

\begin{theorem}\label{periodic}
Suppose that Assumption \ref{b} holds. Consider equation \eqref{main-eq1} for Case 2. If $\sigma$ is bounded and the solution blows up in finite time with positive probability then $b$ satisfies Condition \ref{B}. On the other hand, if $\sigma$ is a positive constant, $u_0$ and $v_0$  
are nonnegative, and Condition \ref{B} holds, then the solution blows up in finite time almost surely
provided that $b$ is also convex.
\end{theorem}

Finally, we study the equation defined on the whole line and show almost sure blow-up under Condition \ref{B}. 

\begin{theorem}\label{line}
Suppose that Assumption \ref{b} holds. Consider equation \eqref{main-eq1} for Case 3 with  $\sigma$ a positive constant and $u_0$ and $v_0$  nonnegative functions. If we suppose that Condition \ref{B} holds, then almost surely, there is no global solution to \eqref{main-eq1}.
\end{theorem}

Theorem \ref{line} complements the main results of \cite{Millet-Sanz} when specialized to our context. Indeed, if $b(x)=\vert x\vert (\log_+ \vert x\vert)^{\delta})$ with $\delta>0$ and $\log_+(z):=\log(z \vee e), z \geq 0$, then Condition \ref{B} holds if and only if $\delta>2$. Observe that in this case, the Osgood condition (\ref{osgood}) holds if and only if $\delta>1$. Thus, a higher power of $\delta$ is needed  for the solution to wave equation to blow up in finite time, compared to the heat equation.

The strategy behind our proofs follows that of \cite{Foondun-Nualart} but with significant differences.  The study of the wave equation is more complicated partly due to the fact that the Green functions are not well behaved.
The first step consists in extending the  results in \cite{leonvilla} to the integral equations associated to  the wave equation. Namely, we show that under Assumption \ref{b}, the integral Condition \ref{B} is necessary and  sufficient for the  blow up in finite time for the solution to the  integral equation given by 
\begin{equation*}
X_t=A+B t+\int_0^t (t-s)\,  b(X_s)\,\d s+g(t),
\end{equation*}
where $A$  and $B$ are nonnegative constants and g is a continuous function  on $[0, \infty)$ satisfying
\begin{equation} \label{g}
\limsup_{t\rightarrow \infty} \inf_{0\leq h\leq 1} g(t+h)=\infty.
\end{equation}
This is achieved by  using comparison theorems for integral equations of the same kind with $g=0$.
Then, in order to show condition (\ref{g}) for the stochastic wave equation,  as opposed to \cite{Foondun-Nualart} where the law of the iterated logarithm for the bi-fractional Brownian motion was used, we have to resort to the theory of Gaussian processes to prove some key estimates. For the case with Dirichlet boundary conditions the use of the support theorem for Brownian motion differs also from the techniques applied in \cite{Foondun-Nualart}.

It is also important to make some comparisons with blow-up results for deterministic wave equations, that is, when $\sigma=0$.  For bounded domains in $\R^n$ with homogeneous Dirichlet boundary conditions, this has been addressed in \cite{G73} where it was shown that instead of Condition \ref{B}, we should have the following integrability for the solution to blow up in finite time
\begin{equation*}
\int_\alpha^{\infty} \frac{1}{[\mu \alpha^2+\beta^2-\mu s^2+2\int_\alpha^s b(r) \, \d r]^{1/2}} \d s <\infty,
\end{equation*}
where $\mu$ is the first eigenvalue of the Dirichlet Laplacian and $\alpha$ and $\beta$ are strictly positive constants depending on the initial conditions. The extra condition that $b$ is convex and $b(s)-\mu s$ is non-decreasing for $s\geq \alpha$ is also imposed. This means that in the deterministic case, we need to the initial conditions to be large enough for the solution to blow up. This is the not the case for corresponding stochastic version. Condition \ref{B} is independent of the domain and of the initial conditions.  This agrees with the intuition that noise pushes the solution high enough for blow-up to occur. The extension to the whole space $\R^n$ in the deterministic  case is also considered in \cite{G73}, where similar integrability conditions above are shown to be sufficient for the solution to blow up in finite time under similar assumptions on the initial condition and $b$.

The rest of the paper is organized as follows. In Section 2 we present the preliminary results explained above needed for the proofs of the main theorems which are given in Section 3.

\section{Preliminary results}

\subsection{Blow up for a class of integral equations}

Suppose that Assumption \ref{b} holds. Consider the following integral equation
\begin{equation} \label{yy}
y(t)=\alpha+\beta (t-t_0)+\int_{t_0}^t (t-s) b(y(s))\,\d s, \quad t \geq t_0,
\end{equation}
where $\alpha>0$ and $\beta>0$.  By Picard-Lindel\"of theorem, equation (\ref{yy}) admits a unique solution up to its blow up time given by
\begin{equation*}
T:=\sup \{t>0: |y(t)|<\infty \},
\end{equation*}
where $\sup \emptyset :=\infty$. We say that the solution blows up in finite time if $T<\infty$. 
\begin{lemma} \label{lem}
Under the above conditions, $T=T(\alpha,\,\beta)$, where $T(\alpha, \beta)$ is given in \textnormal{(\ref{integrability})}.
\end{lemma}

\begin{proof}
We differentiate (\ref{yy}) once to obtain 
$$
y'(t)=\beta+\int_{t_0}^t b(y(s)) \, \d s.
$$
Differentiating again yields  the second order ordinary differential equation (ODE)
$$
y''(t)y'(t)= b(y(t))y'(t) \quad t\geq t_0
$$
with $y(t_0)=\alpha$ and $y'(t_0)=\beta$. The above equation is equivalent to
$$
y'(t)^2-y'(t_0)^2=2\int_{t_0}^t b(y(s)) \,  \d y(s)=2\int_{\alpha}^{y(t)} b(r) \, \d r.
$$
We next write the above as $y'(t)=F(y(t))$ with $F(y(t))=[\beta^2+2\int_{\alpha}^{y(t)} b(r) \, \d r]^{1/2}.$ Then, by the Osgood condition for first order ODEs, $y(t)$ blows up and the blow-up time is given by $T(\alpha,\beta).$
\end{proof}
\begin{remark} \label{22}
We now show that if $T(\alpha,\beta)<\infty$, for some $\alpha,\beta>0$, then $T(\alpha,\beta)<\infty$, for all $\alpha,\beta>0.$ Let $\beta_2\geq \beta_1$, we clearly have $T(\alpha,\beta_2)\leq T(\alpha,\beta_1).$ Now we write
% $\beta_1=\beta_2\left(\frac{\beta_1}{\beta_2}\right)$ 
 \begin{align*}
\beta_1^2+2\int_\alpha^sb(s)\,\d s&=\left(\frac{\beta_1}{\beta_2}\right)^2[\beta_2^2+2\left(\frac{\beta_2}{\beta_1}\right)^2\int_\alpha^sb(s)\,\d s]\\
&\geq\left(\frac{\beta_1}{\beta_2}\right)^2[\beta_2^2+2\int_\alpha^sb(s)\,\d s]
\end{align*}
and therefore $T(\alpha,\beta_1)\leq\frac{\beta_2}{\beta_1} T(\alpha,\beta_2)$. Thus, $T(\alpha,\beta_1)$ is finite if and only if $T(\alpha,\beta_2)$ is finite.  We now let $\alpha_2\geq \alpha_1$. We have $T(\alpha_2,\beta)\leq T(\alpha_1,\beta)$. We now suppose that $T(\alpha_2,\beta)<\infty$. Since $b$ is nonegative,
we have
\begin{align*}
\int_{\alpha_2}^{\infty} \frac{1}{[\beta^2+2\int_{\alpha_1}^{\alpha_2}b(r) \, \d r+2\int_{\alpha_2}^s b(r) \, \d r]^{1/2}} \d s \leq  \int_{\alpha_2}^{\infty} \frac{1}{[\beta^2+2\int_{\alpha_2}^s b(r) \, \d r]^{1/2}} \d s < \infty
\end{align*}
and
$$
\int_{\alpha_1}^{\alpha_2} \frac{1}{[\beta^2+2\int_{\alpha_1}^s b(r) \, \d r]^{1/2}} \d s < \infty.
$$
This means that $T(\alpha_1,\beta)<\infty.$ This finishes the proof.
\end{remark}
We will also need the following comparison result. Let $\tilde{y}(t)$ satisfying the following integral inequality
\begin{equation*} 
\tilde{y}(t)\leq \tilde{\alpha}+\tilde{\beta} (t-t_0)+\int_{t_0}^t (t-s) b(\tilde{y}(s))\,\d s, \quad t \geq t_0,
\end{equation*}
where $\tilde{\alpha}>0$ and $\tilde{\beta}>0$. 
\begin{proposition} \label{comp}
Let $y(t)$ denote the solution to \eqref{yy} and $\tilde{y}(t)$ as above. If $\alpha>\tilde{\alpha}$ and $\beta> \tilde{\beta}$, then $y(t)\geq \tilde{y}(t)$ up to blow-up time. 
\end{proposition}
\begin{proof}
The proof is similar to that of Lemma 2.1 in \cite{leonvilla}. Let 
$$T=\sup\left\{ t>0: |y(t)|\wedge |\tilde{y}(t)|<\infty\right\}$$ 
and consider the set 
\begin{equation*}
A=\{t\in(t_0,\,T]: y(s)\geq \tilde{y}(s)\quad \text{for} \quad s\in(t_0,t] \}.
\end{equation*}
It is straightforward to see that this set is non-empty. We need to show that the supremum of this set is equal to $T$. Let $t_0<T_1<T$. We will prove that such a $T_1$ cannot be the supremum. We have for $t>0$,
\begin{align*}
y(T_1+t)-\tilde{y}(T_1+t)&=\alpha-\tilde{\alpha}+(\beta-\tilde{\beta})(T_1+t)+\int_{t_0}^{T_1+t}(T_1+t-s)(b(y(s))-b(\tilde{y}(s))\,\d s\\
&\geq \alpha-\tilde{\alpha}+(\beta-\tilde{\beta})(T_1+t)+\int_{T_1}^{T_1+t}(T_1+t-s)(b(y(s))-b(\tilde{y}(s))\,\d s.
\end{align*}
By the continuity of the integral, the last term tends to zero as $t$ gets smaller. Hence for small $t$, we have $y(T_1+t)\geq\tilde{y}(T_1+t)$ and therefore $T_1$ cannot be the supremum. This completes the proof.
\end{proof}

\begin{remark} \label{remc}
We have the reverse of the above. If instead we had  $\tilde{y}(t)$ satisfying the following integral inequality
\begin{equation*} 
\tilde{y}(t)\geq \tilde{\alpha}+\tilde{\beta} (t-t_0)+\int_{t_0}^t (t-s) b(\tilde{y}(s))\,\d s, \quad t \geq t_0,
\end{equation*}
but with $\alpha<\tilde{\alpha}$ and $\beta<\tilde{\beta}.$ We would then have $y(t)\leq\tilde{y}(t)$ up to blow-up time and the proof follows exactly as that of Proposition \ref{comp}.
\end{remark}

We next consider the following assumption. 
\begin{assumption}\label{C}
$g: [0,\,\infty) \rightarrow \R$ is a continuous function such that 
\begin{equation*}
\limsup_{t\rightarrow \infty} \inf_{0\leq h\leq 1} g(t+h)=\infty.
\end{equation*}
\end{assumption}

The proof of the following proposition follows along the same lines as that of Proposition 2.2 of \cite{Foondun-Nualart}, using Lemma \ref{yy} and the comparison principle above.
\begin{proposition} \label{p1}
Let $b:\R  \rightarrow \R^+$ as above, $A, B \geq 0$, and suppose that Assumption \ref{C} holds.  Then the solution to the integral equation
\begin{equation}\label{ODE}
X_t=A+B t+\int_0^t (t-s)\,  b(X_s)\,\d s+g(t)
\end{equation}
blows up in finite time if and only if Condition \ref{B} holds.
\end{proposition}

 \begin{proof}
 Suppose that the solution blows up at time $T<\infty$. Since $g$ is continuous, we can set
 \begin{equation*}
M:=\sup_{s \in [0,T]}|g(s)|.
 \end{equation*}
Let $t \in [0,T] $. As $b$ is nonnegative,  \eqref{ODE} gives
\begin{equation*}
X_t\leq A +B t +M+\int_0^t (t-s) \, b(X_s)\,\d s.
\end{equation*}
The nonnegativity of $b$ together with the continuity of $g$ imply that $X_t$ can only blow up to {\it positive infinity}. Let $Y_t=A +(B+1) t+M+1+\int_0^t (t-s)\,  b(Y_s)\,\d s$. Then by the comparison result given by Proposition \ref{comp}, we have $X_t\leq Y_t$ on $[0,\,T]$.  But since $X_t$ blows up at time $T$, $Y_t$ should also blow up by time $T$. By Lemma \ref{lem}, $b$ satisfies Condition \ref{B}, that is, $T(A+M+1, B+1)<\infty$ and hence $T(\alpha,\beta)<\infty$ for all $\alpha>0$ and $\beta>0$. 

We now assume that $T(\alpha, \beta)<\infty$ for some $\alpha, \beta>0.$  Let $\{ t_n\}_{n=1}^\infty$ be some sequence which tends to infinity. The nonnegativity of $b$ implies that
 \begin{align*}
 X_{t+t_n}&\geq A +B t+\int_{t_n}^{t+t_n} (t-s) \; b(X_s)\,\d s+g(t+t_n)\\
 &\geq A+B t+\int_{0}^{t} (t-s) \, b(X_{s+t_n})\,\d s+g(t+t_n)\\
 &\geq A+\frac{1}{2}\inf_{0\leq h\leq 1}g(h+t_n)+(B +\frac{1}{2}\inf_{0\leq h\leq 1}g(h+t_n))t+\int_{0}^{t}(t-s) \; b(X_{s+t_n})\,\d s,
\end{align*}
where the last inequality holds whenever $0\leq t \leq 1$. 

We now set $\alpha_n:=A+\frac{1}{4}\inf_{0\leq h\leq 1}g(h+t_n)$ and $\beta_n:=B +\frac{1}{4}\inf_{0\leq h\leq 1}g(h+t_n)$, where $n$ is taken large enough so that $\inf_{0\leq h\leq 1}g(h+t_n)>0$. The comparison principle described in Remark \ref{remc} implies that $X_{t+t_n}\geq Z_t$, where
\begin{equation*}
Z_t=\alpha_n+\beta_n t+\int_{0}^{t}(t-s) b(Z_s)\,\d s.
\end{equation*}
Since we are assuming that $T(\alpha, \beta)<\infty$, we can take $n$ large enough so that $T(\alpha_n,\beta_n)<1$ which means that that the blow-up time of $Z_t$ is strictly less than 1. Now since $X_{t+t_n}\geq Z_t$ we have that the blow-up of $X_t$ is finite and the proof is complete.
%Since we are assuming that $X_t$ does not blow up in finite time, the blow up time of $Z_t$ has to be greater than 1, which implies that 
%\begin{equation*}
%\int_{\alpha_n}^\infty \frac{1}{[\beta^2+2\int_{\alpha_n}^s b(r) \d r]^{1/2}}\,\d s>1\quad\text{with}\quad \alpha_n:=\frac{1}{2}(a+\inf_{0\leq h\leq 1}g(h+t_n)).
%\end{equation*}
%But from Assumption \ref{B}, we can find a sequence $t_n\rightarrow \infty$ such that $\alpha_n=\rightarrow \infty$. This contradicts Assumption \ref{B}
%and the proof is complete.
\end{proof}
We will need to check Assumption \ref{C} for Case 3 (see Section 2.3). Recall that for the  heat equation, the law of iterated logarithm for the bi-fractional Brownian motion  was used, see \cite{Foondun-Nualart}. Instead, here we will use a limit theorem for Gaussian processes proved in \cite[Theorem 4]{W70}. Let $(X(t), t \in \R^+)$ be a real separable centered Gaussian process. We denote covariance function by 
$r(t,s)=\E(X(t) X(s))$ and standard deviation by $v(t)=\sqrt{r(t,t)}$. The continuous correlation function is given by 
$$
\rho(s,t)=\frac{r(t,s)}{v(t) v(s)}.
$$
Consider the following assumptions.
\begin{assumption} \label{A1}
\begin{enumerate}
\item[\textnormal{(i)}] Suppose that there  exist positive constant $c,T$ and $\delta$ such that
$
\rho(t, t+h) \leq 1-c(h/t)^{\alpha},
$
for all $t$ and $h$ such that $t>T$ and $0<h/t<\delta$, and for all $t$ and $s$ such that
$h/t>\delta$ and $t>T$, $\rho(t, t+h)<1-c\delta^{\alpha}$, for some $\alpha \in (0,2]$.

\item[\textnormal{(ii)}] Suppose that 
$
\lim_{s \rightarrow \infty} \rho(t, ts) \log s=0,
$
uniformly with respect  to $t$.
\end{enumerate}
\end{assumption}

\begin{theorem} \label{tw}
Suppose that Assumption \ref{A1} holds. 
Then, almost surely,
\begin{equation} \label{li}
\limsup_{t \rightarrow \infty} \frac{X(t)}{v(t) \sqrt{2\log \log t}}>1.
\end{equation}
\end{theorem}
Set 
\begin{equation}\label{G}
G(t):=\int_0^tB_s\,\d s,
\end{equation}
where  $(B_t)_{t \geq 0}$ is a Brownian motion. Recall that $G(t)$ is a Gaussian process with mean zero and variance $t^3/3$. The following is \cite[Corollary 2, p.238]{W70}.
\begin{proposition}\label{LIL}
Almost surely, 
\begin{equation*}
\limsup_{t\rightarrow \infty}\frac{G(t)}{\sqrt{\frac{2}{3}t^3\log \log t}}=1.
\end{equation*}
\end{proposition}
 
Following similar ideas as in \cite[Section 4.1] {leonvilla} we get the following.
\begin{proposition} \label{assumpC}
The process $G(t)$ satisfies Assumption \ref{C} almost surely.
\end{proposition}

\begin{proof}
  Observe that 
$$\E((G(t)-G(s))^2)=(t-s)^3/3+s(t-s)^2 \leq t (t-s)^2 \quad \text{for all} \quad s \leq t.$$
Then, appealing to \cite[Lemma 5.2]{Carmona},
 for all $p \geq 1$, there exists a constant  $A_p>0$ such that 
 for any integer $n \geq 1$,
 $$
  \E\left[\sup_{s,t \in [n,n+2]}\vert G(t)-G(s) \vert^p \right] \leq A_p n^{p/2}.
  $$ 
  Observe that we are applying \cite[Lemma 5.2]{Carmona} with $C(T)=n+2\leq 3n$ and $(t-s) \leq 2$.
 Therefore, for $p>1$, we get that
 \begin{equation}\label{li1} \begin{split}
 \E \left[\sum_{n=1}^{\infty} \sup_{s,t \in [n,n+2]}
  \frac{\vert G(t)-G(s)\vert^p}{(n^{3/2}\sqrt{\log \log n})^p}\right] 
  &\leq \sum_{n=1}^{\infty} \frac{A_p n^{p/2}}{(n^{3/2}\sqrt{\log \log n})^p} <\infty.
 \end{split}
 \end{equation}
 Next let $\omega \in \Omega$ such that (\ref{li}) and (\ref{li1}) hold. Then, we can write
 \begin{align*}
&\inf_{h \in [0,1]} G(t+h)(\omega) \geq G(t)(\omega)+\inf_{h \in [0,1]} \left(-\vert G(t+h)(\omega)-G(t)(\omega)\vert\right) \\
  & \qquad \geq  \frac{G(t)(\omega)}{t^{3/2}\sqrt{\log \log t}} t^{3/2}\sqrt{\log \log t}-\sup_{h \in [0,1]} \frac{\vert G(t+h)(\omega)-G(t)(\omega)\vert}{[t^{3/2}]\sqrt{\log \log [t]}} [t^{3/2}]\sqrt{\log \log [t]} \\
 & \qquad \geq  \frac{G(t)(\omega)}{t^{3/2}\sqrt{\log \log t}} t^{3/2}\sqrt{\log \log t}-\frac14 [t^{3/2}]\sqrt{\log \log [t]} ,
  \end{align*} 
  where $[t]$ is the integer part of $t$. Thus, using (\ref{li}) the proof is completed.
\end{proof}

\subsection{Estimates for Case 1}
Set
\begin{equation}\label{M}
M(t):=\kappa \int_0^t\int_0^1\int_0^1 G_1(t-s, x,y)\varphi_1(x)W(\d s\,\d y)\,\d x,
\end{equation}
where  $\kappa^{-1}:=\int_0^1 \varphi_1(x)\,\d x$. 

We will need the following support theorem taken from \cite{PTA}.

\begin{theorem}\label{support}
If $f:[0,\,1]\rightarrow \R$ is continuous with $f(0)=0$, then for $\epsilon>0$, we have 
\begin{equation*}
\P(\sup_{0\leq t\leq 1}|B_t-f(t)|<\epsilon)>\eta,
\end{equation*}
where $\eta$ depends only on $\epsilon$ and on the modulus of continuity of $f$.
\end{theorem}
We then have the following which is a consequence of the above theorem.
\begin{proposition}\label{deviation}
Fix $L>0$,  then with a positive probability, the following holds 
\begin{equation*}
M(t)\geq L\quad \text{for all}\quad t\in[\frac{1}{16},\,\frac{3}{16}].
\end{equation*}

\end{proposition}

\begin{proof}
We start by rewriting $M(t)$ as follows, 
\begin{align*}
M(t)&=\frac{\kappa}{\pi} \int_0^t\int_0^1 \sin(\pi (t-s))\varphi_1(y) \,W(\d y \,\d s)=\frac{1}{\pi}\int_0^t \sin(\pi (t-s))\,\d B_s,
\end{align*}
where  $(B_t)_{t \geq 0}$  is a  Brownian motion.
Integrating by  parts, we can further write
\begin{equation*}
M(t)=\int_0^t\cos(\pi (t-s))B_s\,\d s.
\end{equation*}
If we now choose $f(t)=32\tilde{L}t$, Theorem \ref{support} tells us that the event
\begin{equation*}
A:=\{\sup_{0\leq t\leq 1}|B_t-32\tilde{L}t|\leq \epsilon\}
\end{equation*}
occurs with a positive probability. Therefore, for any $\omega\in A$ and $t\in [\frac{1}{16}, \frac{3}{16}]$, we have 
\begin{align*}
M(t)&= \int_0^{t_0} \cos(\pi (t-s))B_s\,\d s+ \int_{t_0}^{1/32}\cos(\pi (t-s))B_s\,\d s+\int_{1/32}^t\cos(\pi (t-s))B_s\,\d s.
\end{align*}
We now choose $t_0$ so that for $t\geq  t_0$, we have $B_t\geq 0$ for $\omega\in A$. From Theorem \ref{support}, we see that $t_0=\frac{\epsilon}{32\tilde{L}}.$ For $0\leq t \leq t_0$, we have $B_t\geq -\epsilon$. Since the $\cos\pi(t-s)$ is bounded below by a strictly positive number for $t\in [\frac{1}{16}, \frac{3}{16}]$ and $s\in [0,\,t]$, we can ignore the second term of the above display and write
\begin{align*}
M(t)&\geq -c_1\frac{\epsilon^2}{32 \tilde{L}}+c_2(t-\frac{1}{32})(32\tilde{L}t-\epsilon)\geq -c_1\frac{\epsilon^2}{32 \tilde{L}}+c_2\frac{1}{32}(2\tilde{L}-\epsilon).
\end{align*}
The result is proved once we choose $\tilde{L}$ appropriately. 
\end{proof}
\subsection{Estimates for Case 3}
For Case 3, which is the real line case, we consider
\begin{equation} \label{gtx}
g(t,\,x):=\int_0^{t}\int_{\R} G_3(t-s,\,x-y)W(\d y\,\d s).
\end{equation}
For fixed $x \in \R$, $\{g(t,x)\}_{t \geq 0}$  is a centered Gaussian process with covariance given by
$$
\E (g(t,x) g(s,x))=\frac{(s \wedge t)^2}{4}, \qquad s,t \geq 0.
$$
This follows from the fact that for $s \leq t$, we have 
\begin{equation*} \begin{split}
 \E (g(t,x) g(s,x))&=\int_0^s  \int_{\R}  G_3(t-r,\,x-y) G_3(s-r,\,x-y)  \d y \, \d r\\
&=\frac14 \int_0^s \int_{\R}  {\bf 1}_{\vert  x-y \vert < s-r} \d y \, \d r =\frac{s^2}{4}.
\end{split}
\end{equation*}
In particular, this Gaussian process satisfies the  following consequence of Theorem \ref{tw}.
\begin{proposition} \label{prop4}
For fixed $x \in \R$, almost surely,
$$
\limsup_{t \rightarrow \infty} \frac{\sqrt{2}g(t,x)}{t\sqrt{2\log \log t}}>1.
$$
\end{proposition}
\begin{proof}
We will apply Theorem \ref{tw}. As $v(t)=t/2$, it suffices to check Assumption \ref{A1}. We start with (i). If $0<h/t<\delta$ then
$
\rho(t, t+h)=1-\frac{\frac{h}{t}}{1+\frac{h}{t}}<1-\frac{\frac{h}{t}}{1+\delta},
$
thus the  first part of (i) holds with $c=\frac{1}{1+\delta}$ and  $\alpha=1$.
Moreover, if $h/t>\delta$, we have
$
\frac{1}{1+\frac{h}{t}}<\frac{1}{1+\delta}=1-\frac{\delta}{1+\delta},
$
thus the second part of (i) follows as well.  Finally, if $s>1$ we have $ \rho(t, ts)=\frac{1}{s},$
so (ii) is  also satisfied. The proof of the theorem is now complete. 
\end{proof}

We will need the following moment estimates.
\begin{proposition} \label{holder2bb}
For all $t,h>0$ and $z \in \R$,
 \begin{equation*}  \begin{split}
\sup_{x \in \R}\E\left[\vert g(t,x+z)-g(t,x) \vert^2\right] &\leq  
|z| t,\\
\sup_{x \in \R} \E\left[\vert g(t+h,x)-g(t,x) \vert^2\right] &\leq  
h(t+h).
\end{split}
 \end{equation*} 
 \end{proposition}
 
 \begin{proof}
 Using expression (\ref{gtx}) and changing variables, we get that
 \begin{equation*}
 \begin{split}
 \E\left[\vert g(t,x+z)-g(t,x) \vert^2\right]&=\frac{1}{4} \int_0^t \int_{\R} \vert {\bf 1}_{\{\vert z+y\vert <s\}}-
 {\bf 1}_{\{\vert y\vert <s\}}\vert^2 \, \d y\,  \d s \\
 &=\frac{1}{4} \int_0^t \int_{\R} {\bf 1}_{\{s-\vert  z\vert  \leq \vert y\vert <s\}} \, \d y \, \d s\leq  \vert z \vert t.
 \end{split}
 \end{equation*}
 Similarly, adding and subtracting the term $\int_0^t G_3(t+h-s,x-y)  W(dy,ds)$, we  get
 \begin{equation*}
 \begin{split}
 &\E\left[\vert g(t+h,x)-g(t,x) \vert^2\right]\\
 &\qquad \leq \frac{1}{2} 
 \int_0^{h} \int_{\R} {\bf 1}_{\{\vert y\vert <s\}} \, \d y \, \d s+\frac 12
 \int_0^t \int_{\R} \vert {\bf 1}_{\{\vert y\vert <s+h\}}-
 {\bf 1}_{\{\vert y\vert <s\}}\vert^2 \, \d y\,  \d s\\
 &\qquad = h^2+\frac 12
 \int_0^t \int_{\R}  {\bf 1}_{\{s \leq \vert y\vert <s+h\}} \, \d y\,  \d s=h^2+th.
 \end{split}
 \end{equation*}
 \end{proof}
 
 As a consequence of Proposition \ref{holder2bb}, we obtain the following estimate. The proof uses an isotropic Kolmogorov continuity theorem obtained in \cite{DKN}, similarly as in the proof of Lemma 4.5 in \cite{DKN}.
 \begin{proposition} \label{supn}
 For all $p \geq 2$, there exists a constant  $A_p>0$ such that 
 for any integer $n \geq 1$,
 $$
  \E\left[\sup_{s,t \in [n,n+2], x,y \in [0,1]}\vert g(t,x)-g(s,y) \vert^p \right] \leq A_pn^{p/2} .
  $$ 
 \end{proposition}
 
 \begin{proof}
 Let  $n \geq 1$ be a fixed integer.
We plan to apply Proposition A.1 and Remark A.2(a) in \cite{DKN} with 
 $S=\{ [n,n+2]\times [0,1]\}$, $\rho((t,x),(s,y))=(\vert t-s \vert+\vert x-y \vert)^{1/2}$,
 $\mu(\d t\d x)=\d t\d x$, $\Psi(x)=e^{\vert x\vert}-1$, and  $p(x)=\sqrt{n}x$. 
We set
$$
\mathcal{C}:=\int_{S} \int_{S} \exp \left( \frac{\vert g(t,x)-g(s,y) \vert}{\sqrt{n}\rho((t,x),(s,y)) }\right)\d t \, \d x \, \d s\, \d y .
$$
By Proposition \ref{holder2bb}, for any $(t,x)$ and  $(s,y)$ in $S$, we have that
$$
(\E\left[\vert g(t,x)-g(s,y) \vert^2\right])^{1/2} \leq 3 \sqrt{n}\rho((t,x),(s,y)).
$$
Then, it easily follows that
 $$
 \E[\mathcal{C}]=\int_{S} \int_{S} \E \left[\exp \left( \frac{\vert g(t,x)-g(s,y) \vert}{\sqrt{n}\rho((t,x),(s,y)) }\right)\right]\d t \, \d x \, \d y \,\d s\leq e^3 \int_{S} \int_{S} \d t \, \d x \, \d y \,\d s = 4e^3.
 $$
  Observe that $$\{s,t \in [n,n+2], x,y \in [0,1]\}\subset \{s,t \in [n,n+2], x,y \in [0,1]: \rho((t,x),(s,y))\leq \sqrt{3}\}$$
Therefore, by Proposition A.1 and Remark A.2(a) with $\epsilon=\sqrt{3}$ and Jensen's inequality, we get that for all $p \geq 2$,
 \begin{equation*}
 \begin{split}
  &\E\left[\sup_{s,t \in [n,n+2], x,y \in [0,1]}\vert g(t,x)-g(s,y) \vert^p \right] \\
  &\qquad \leq 10^p \E \left[ \sup_{t \in [n,n+2], x \in [0,1]} \left(\int_0^{2\sqrt{3}} 
  \sqrt{n} \, \d u \ln \left( 1+\frac{\mathcal{C}}{[\mu(B_{\rho}((t,x),u/4))]^2}\right)\right)^p \right] \\
 &\qquad = 10^p n^{p/2}\E \left[  \left(\int_0^{2\sqrt{3}} 
   \d u \ln \left( 1+\frac{\mathcal{C}}{c u^8}\right)\right)^p \right] \\
  &\qquad \leq A'_p n^{p/2}\int_0^{2\sqrt{3}} 
  \d u \ln^p \left( 1+\frac{\mathcal{\E[ C]}}{c u^8}\right) \\
  &\qquad \leq A'_p n^{p/2}\int_0^{2\sqrt{3}} 
  \d u \ln^p \left( 1+\frac{4 e^3}{c u^8}\right)=A_p n^{p/2},
 \end{split}
 \end{equation*}
  where $B_{\rho}((t,x),r)$ denotes open ball of radius $r>0$ and  center $(t,x) \in S$ with distance given by $\rho$ and $A_p$ and $A'_p$ are positive constants independent on $n$.
 This implies the desired result.
 \end{proof}
 
We can now use Proposition \ref{supn} to get the following almost sure result. 
\begin{proposition} \label{supnas}
Almost surely
 $$
   \sup_{s,t \in [n,n+2], x,y \in [0,1]}
  \frac{\vert g(t,x)-g(s,y) \vert}{n\sqrt{\log \log n}} \longrightarrow 0, \quad \text{ as } n \rightarrow \infty.
  $$ 
 \end{proposition}
 \begin{proof}
Applying Proposition \ref{supn} with $p>2$, we obtain
 \begin{equation*} \begin{split}
 \E \left[\sum_{n=1}^{\infty} \sup_{s,t \in [n,n+2], x,y \in [0,1]}
  \frac{\vert g(t,x)-g(s,y) \vert^p}{(n\sqrt{\log \log n})^p}\right] 
  &\leq \sum_{n=1}^{\infty} \frac{A_p n^{p/2}}{(n\sqrt{\log \log n})^p} <\infty.
 \end{split}
 \end{equation*}
 The proof is completed.
 \end{proof}
We are now ready to show a key result behind the  proof of Theorem \ref{line}.
\begin{proposition} \label{p2}
Almost surely, there exists a sequence $t_n \rightarrow \infty$ such that 
 $$
 \inf_{h \in [0,1], x \in [0,1]} g(t_n+h,x) \rightarrow \infty.
 $$
 \end{proposition}
 
 \begin{proof}
 Fix $x_0 \in [0,1]$ and write
 \begin{align*}
&\inf_{h \in [0,1], x \in [0,1]} g(t+h,x)\geq g(t,x_0)+\inf_{h \in [0,1], x \in [0,1]} \left(-\vert g(t+h,x)-g(t,x_0)\vert\right) \\
  &\qquad \geq  \frac{g(t,x_0)}{t\sqrt{\log \log t}} t\sqrt{\log \log t}-\sup_{h \in [0,1],x \in  [0,1]} \frac{\vert g(t+h,x)-g(t,x_0)\vert}{[t]\sqrt{\log \log [t]}} [t]\sqrt{\log \log [t]}.
  \end{align*}
  We apply Propositions \ref{prop4} and \ref{supnas} to conclude.
 \end{proof}

\section{Proofs of Theorems \ref{Bonder-Groissman}, \ref{periodic}, and \ref{line}}
\begin{proof}[Proof of Theorem \ref{Bonder-Groissman}]
We start proving the first implication. Since the solution blows up in finite time with positive probability, we can find a set $\Omega$  satisfying $\P(\Omega)>0$ such that for any $\omega\in \Omega$, we have $\tau(\omega)<\infty$, where $\tau$ is defined in (\ref{tau}).  We fix such an $\omega$ but we won't indicate the dependence on $\omega$ in what follows to simplify the notation.
Consider the mild formulation which is given by 
\begin{equation}\label{mild11} \begin{split}
&u(t,x)=\int_0^1 G_1(t,x,y) v_0(y) \, \d y+\frac{\partial}{\partial t} \left(\int_0^1 G_1(t,x,y) u_0(y) \, \d y \right) \\
&\quad +\int_0^t\int_0^1 G_1(t-s, x,y) \sigma(u(s,y))\,W(\d s\,\d y) +\int_0^t\int_0^1 G_1(t-s, x,y)b(u(s,y))\, \d s\, \d y \\
&=:I_1(t,x)+I_2(t,x)+I_3(t,x)+I_4(t,x).
\end{split}
\end{equation}
We will bound each term in (\ref{mild11}) separately. 
First, since $v_0$ is bounded in $[0,1]$, using (\ref{op}), we have that
$$
\vert I_1(t,x) \vert \leq c t,
 $$
 for some constant $c>0$. 
Similarly, 
\begin{align*}
\vert I_2(t,x) \vert\leq C,
\end{align*}
for some constant $C>0$.
Set
\begin{equation*}
M_{\tau}:=\sup_{(t,x) \in (0, \tau]\times [0,1]}\left|I_3(t,x)\right|.
\end{equation*}
Then, since $\sigma$ is bounded, using Burholder-David-Gundy inequality and (\ref{g1q}), we obtain  that
\begin{equation*} \begin{split}
\E[M_{\tau}]&\leq \E\big[\sup_{(t,x) \in (0, \tau]\times [0,1]}\left|I_3(t,x)\right|\big] \\
&\leq c  \sup_{(t,x) \in (0, \tau]\times [0,1]} \left(\int_0^{t}\int_0^1 G_1^2(t-s,x,y) \, \d y \, \d s\right)^{1/2}<\infty.
\end{split}
\end{equation*}
Hence, $M_{\tau}<\infty$, a.s.
Finally, set $Y_t:=\sup_{x\in [0,\,1]} \vert u(t,\,x)\vert.$ Then, appealing again to (\ref{op}), we obtain that
\begin{align*}
\vert I_4(t,x)\vert   \leq\int_0^t b(Y_s) (t-s)\,\d s.
\end{align*}

Therefore, we have shown that uniformly for all $t \in (0, \tau]$,
\begin{equation} \label{yt}
Y_t\leq C+ct+ M_{\tau} +\int_0^t (t-s) b(Y_s) \,\d s.
\end{equation}
We can now use Proposition \ref{comp} with $\tilde{y}(t)=Y_t$ and $y(t)$ the solution to
$$
y(t)= C+1+(c+1)t+ M_{\tau} +\int_0^t (t-s) b(y(s)) \,\d s,
$$
to obtain that $y(t) \geq \tilde{y}(t)$. Hence, since $\tilde{y}(t)$ blows up  at time $\tau$, $y(t)$ should blow up by time $\tau$.
Then, by Lemma \ref{lem} we conclude that 
Condition \ref{B} holds. This shows the first part  of the theorem.  

We now prove the  second part. Set $g(t):=\kappa\int_0^1 u(t,x)\varphi_1(x)\,\d x$, where $\kappa$ is defined in (\ref{M}).
We restrict $t$ to the interval  $[\frac{1}{16},\,\frac{3}{16}]$ as in Proposition \ref{deviation}. We will use the mild formulation (\ref{mild11}) with $\sigma$ being a positive  constant. We have,
\begin{align*}
\int_0^1I_1(t,x)\varphi_1(x) \, \d x&=\int_0^1 \frac{\sin \pi t}{\pi}\varphi_1(y)v_0(y)\,\d y\geq \frac{2t}{\pi} \int_0^1\varphi_1(y)v_0(y)\,\d y
\end{align*}
and
\begin{align*}
\int_0^1I_2(t,x) \varphi_1(x)\,\d x&=\int_0^1 \cos \pi t \varphi_1(y) u_0(y)\,\d y\geq \cos\left(\frac{3\pi}{16}\right)\int_0^1 \varphi_1(y) u_0(y)\,\d y.
\end{align*}
As $b$ is convex, using Jensen's inequality, we get that
\begin{align*}
\int_0^1 I_4(t,x) \kappa \varphi_1(x)\,\d x&=\int_0^t\int_0^1 \frac{\sin\pi (t-s)}{\pi} \kappa \varphi_1(y) b(u(s,y))\,\d y\,\d s\\
&\geq \frac{2}{\pi} \int_0^t (t-s) b(g(s))\,\d s.
\end{align*}

We combine the above to obtain 
\begin{align*}
g(t)\geq A+ Bt+ \frac{2}{\pi} \int_0^t (t-s) b(g(s))\,\d s+\sigma M(t)\quad \text{for}\quad t\in [\frac{1}{16},\,\frac{3}{16}],
\end{align*}
where  $A,B$ are nonnegative constants and $M(t)$ is given by \eqref{M}.  According to Proposition \ref{deviation}, for any $L>0$, with a positive probability,
on $t\in [\frac{1}{16},\,\frac{3}{16}]$,
\begin{align*}
g(t)\geq L+ (B+1)(t-\frac{1}{16})+ \frac{2}{\pi} \int_{\frac{1}{16}}^t (t-s) b(g(s))\,\d s,
\end{align*}
where  we have used  the  fact that  $\frac{3}{16}-\frac{1}{16}\geq t-\frac{1}{16}$, $b$ is nonnegative, and $L$ can be taken as large as needed.

By the comparison principle explained in Remark \ref{remc}, we conclude that, with positive probability, on $t\in [\frac{1}{16},\,\frac{3}{16}]$, $y(t)\leq g(t)$, where 
$y(t)$ is the solution to 
$$
y(t)= \frac{L}{2}+\frac{B+1}{2}(t-\frac{1}{16})+ \frac{2}{\pi} \int_{\frac{1}{16}}^t (t-s) b(y(s))\,\d s.
$$
We now assume Condition \ref{B}.
Then, we need to make sure that under the assumption that 
\begin{equation*}
\int_{L/2}^\infty\frac{1}{[((B+1)/2)^2+\frac{4}{\pi}\int_{L/2}^s b(r)\,\d r]^{1/2}}\,\d s<\infty,
\end{equation*}
$y(t)$ blows up in a small time. But this follows from the fact that we can always take $L$ to be large enough so that the above integral is small enough so that $y$  has a blow up time in $[\frac{1}{16},\,\frac{3}{16}]$. Then, $g(t)$ will also blow up with positive probability on $[\frac{1}{16},\,\frac{3}{16}]$.
The proof of the theorem is now completed.\end{proof}

\begin{proof}[{Proof of Theorem \ref{periodic}}]
The mild formulation  writes as
\begin{equation*} \begin{split}
&u(t,x)=\frac12 \int_{x-t}^{x+t} v_0(y) \, \d y+\frac12 (u_0(x-t)+u_0(x+t)) \\
&\quad +\int_0^t\int_I G_2(t-s, x-y) \sigma(u(s,y))\,W(\d s\,\d y)+\int_0^t\int_I G_2(t-s, x,y)b(u(s,y))\,\d s\,\d y.
\end{split}
\end{equation*}
The proof of the first part  follows as in Case  1 above. Indeed, since $u_0$, $v_0$, and $\sigma$ are bounded and using  (\ref{g2}), we obtain  that (\ref{yt}) also holds true with $Y_t=\sup_{x\in I} \vert u(t,\,x)\vert$. Thus, using Proposition \ref{comp} and Lemma \ref{lem} we conclude. 

 We next prove the  second part. Set $X_t:=\frac{1}{2\pi} \int_I u(t,x) \, \d x$. We have
$$
\frac{1}{2}\int_I\left(\int_{x-t}^{x+t} v_0(y)\, \d y +u_0(x-t)+u_0(x+t)\right)\, \d x=t \int_I v_0(y)\, \d  y+\int_I u_0(y) \, \d y.
$$
By the stochastic Fubini theorem,
$$
\frac{1}{2\pi}\int_I\int_0^t\int_I G_2(t-s, x-y) \,W(\d s\,\d y)\, \d x=\frac{1}{2\pi}\int_0^t\int_I (t-s) \,W(\d s\,\d y)=:G(t).
$$
Finally, by Jensen's inequality, using the fact that $b$ is convex,
$$
\frac{1}{2\pi}\int_I\int_0^t\int_I G_2(t-s, x,y)b(u(s,y))\,\d s\,\d y\, \d x \geq 
\int_0^t (t-s) b(X_s) ds.
$$
Therefore, we have  proved that for all $t \geq 0$
$$
X_t \geq  A+Bt+\sigma G(t)+\int_0^t (t-s) b(X_s) \, \d s,
$$
for some nonnegative  constants  $A,B$. 
We now note that since $G(t)$ is given by \eqref{G}, it therefore satisfies Proposition \ref{assumpC}.
Thus, we can proceed as in the proof of Proposition \ref{p1}. Choose $\omega \in \Omega$ satisfying Proposition \ref{assumpC}
and let $t_n \rightarrow \infty$ such that $\inf_{0\leq h \leq 1} G(h+t_n)$ goes to infinity. Using the nonnegativity of $b$, we can write
 \begin{align*}
 &X_{t+t_n}\geq A +B t+\int_{t_n}^{t+t_n} (t-s) \; b(X_s)\,\d s+\sigma G(t+t_n)\\
 &\quad \geq A+\frac{1}{2}\inf_{0\leq h\leq 1} \sigma G(h+t_n)+(B +\frac{1}{2}\inf_{0\leq h\leq 1} \sigma G(h+t_n))t+\int_{0}^{t}(t-s) \; b(X_{s+t_n})\,\d s,
\end{align*}
for $t \in [0,1]$. 
Set $\alpha_n:=A+\frac{1}{4}\inf_{0\leq h\leq 1}\sigma G(h+t_n)$ and $\beta_n:=B +\frac{1}{4}\inf_{0\leq h\leq 1}\sigma G(h+t_n)$, where $n$ is taken large enough so that $\inf_{0\leq h\leq 1}G(h+t_n)>0$. Remark \ref{remc} implies that $X_{t+t_n}\geq Z_t$, where
\begin{equation*}
Z_t=\alpha_n+\beta_n t+\int_{0}^{t}(t-s) b(Z_s)\,\d s.
\end{equation*}
We next assume that Condition \ref{B} holds and we take $n$ large enough so that $T(\alpha_n,\beta_n)<1$, that is, the blow-up time of $Z_t$ is strictly less than 1. Since $X_{t+t_n}\geq Z_t$ this implies that $X_t$ blows up in finite time,
which concludes the second part of the theorem. 
\end{proof}

\begin{proof}[Proof of Theorem \ref{line}]
In this case, the mild formulation  writes as
\begin{equation}\label{mbis}\begin{split}
u(t,x)=I(t,x)+\sigma g(t,x)+\frac12 \int_0^t\int_{\vert x-y \vert<t-s} b(u(s,y))\,\d s\,\d y,
\end{split}
\end{equation}
where
$$
I(t,x):=\frac12 \int_{x-t}^{x+t} v_0(y) \, \d y+\frac12 (u_0(x-t)+u_0(x+t))
$$
and $g(t,x)$ is defined in (\ref{gtx}).
Let $\{t_n\}$ be a sequence of positive numbers which we are going to choose later. From (\ref{mbis}) and the nonnegativity of the function $b$, we obtain 
\begin{equation*}\begin{split}
u(t+t_n,\,x)
\geq I(t+t_n, \, x) +\sigma g(t+t_n,x)+\frac12\int_0^{t}\int_{\vert x-y \vert<t-s} b(u(s+t_n,\,y))\,\d y\,\d s.
\end{split}
\end{equation*}
Hence by Proposition \ref{p2}, we can find a sequence $t_n \rightarrow \infty$ so that $g(t+t_n,\,x)$ (and thus $u(t+t_n,\,x)$) are  positive for all $0\leq t\leq 1$ and $x \in (0,1)$. On the  other hand, since $b$ is nondecreasing, for fixed $x \in (0,1)$, we get that
\begin{align*}
\int_0^{t}\int_{\vert x-y \vert<t-s} b(u(s+t_n,\,y))\,\d y\,\d s &\geq \int_0^t b\left(Y_s\right)
\int_{\{\vert x-y \vert<t-s \}\cap \{y \in (0,1)\}}\,\d y\,\d s\\
& \geq \int_0^t (t-s)\, b\left(Y_s\right)\,\d s,
\end{align*}
where $Y_t:=\inf_{y \in (0,1)} u(t+t_n,\,y)$. Combining the above estimates we obtain 
$$
Y_t\geq \inf_{0\leq h\leq 1, x \in (0,1)}\{I(h+t_n,x)+ \sigma g(h+t_n,\,x)\}+\int_0^t (t-s) \, b(Y_s)\,\d s.
$$
We now choose $\omega$ as in Proposition \ref{p2}, and we can therefore find a sequence $t_n\rightarrow \infty $ such that $\inf_{0\leq h\leq 1, x \in (0,1)}g(h+t_n,\,x)$ goes to infinity. Using a similar argument as in  the Proposition \ref{p1} as we did for  Case 2 above, we conclude.
\end{proof}

\section{Appendix}

In this section we provide some properties of the three Green kernels $G_1$, $G_2$, and $G_3$ defined in the Introduction. First observe that the kernels $G_2$ and $G_3$ are non-negative functions, while $G_1$ is a real-valued function on $\R_+ \times [0,1] \times [0,1]$. In fact, $G_1$ has the alternative expression
\begin{equation} \label{ee1}
G_1(t,x,y)=\sum_{n \in \mathbb{Z}} \left({\bf 1}_{\{\vert y-x-2n\vert \leq t \}}-{\bf 1}_{\{\vert y+x-2n\vert \leq t  \}} \right),
\end{equation}
see \cite{Caba}.
This makes the study of the wave equation in $[0,1]$ more complicated. Observe that for all $y \in \R$, we have $$\int_{\R} G_3(t, x-y)\, \d x=t.$$
It is easy to check that $G_2$ satisfies the same property. Indeed, for all $y \in I$,
\begin{equation}\label{g2} \begin{split}
\int_I G_2(t, x-y) \, \d x&= \frac{1}{2} \sum_{n \in \mathbb{Z}}
\int_I {\bf 1}_{\{ \vert x-y +2n\pi\vert <t\}} \, \d x\\
& = \frac{1}{2} \sum_{n \in \mathbb{Z}}
\int_{2n\pi}^{2\pi(n+1)} {\bf 1}_{\{ \vert x-y \vert <t\}} \, \d x=\frac{1}{2} 
\int_{\R} {\bf 1}_{\{ \vert x-y \vert<t\}} \, \d x =t.
\end{split}
\end{equation}
On the other hand, using (\ref{ee1}) we get that for all $y \in [0,1]$,
\begin{equation} \label{op}\begin{split}
\int_0^1 G_1(t, x,y) \, \d x&\leq  \sum_{n \in \mathbb{Z}} \int_0^1 {\bf 1}_{\{\vert y-x-2n\vert \leq t \}} \, \d x
= \sum_{n \in \mathbb{Z}} \int_{2n}^{1+2n} {\bf 1}_{\{\vert y-x\vert \leq t \}} \, \d x\\
&=\frac12 \int_{\R} {\bf 1}_{\{ \vert x-y \vert<t\}} \, \d x =t.
\end{split}
\end{equation}
Similarly, we have  that for all $T>0$ and $y \in [0,1]$,
\begin{equation} \label{g1q}
\sup_{(t,y) \in [0,T] \times [0,1]} \int_0^1 G^2_1(t, x,y) \, \d x<\infty.
\end{equation}

\bibliography{Foon-Nual}
\end{document}